\documentclass[12pt]{amsart}

\usepackage{amsmath}
\usepackage{amssymb}
\usepackage{eucal}

\newtheorem{theorem}{Theorem}[section]

\newtheorem{prop}[theorem]{Proposition}
\theoremstyle{definition}

\theoremstyle{remark}

\allowdisplaybreaks
\numberwithin{equation}{section}

\begin{document}

\title[Periods of second kind differentials of $(n,s)$-curves]
{Periods of second kind differentials of $(n,s)$-curves}

\author{J. C. Eilbeck}
\address{ Department of Mathematics and Maxwell Institute, Heriot-Watt
  University, Edinburgh, EH14 4AS, UK}
\email{J.C.Eilbeck@hw.ac.uk}

\author{K. Eilers}
\address{Faculty of Mathematics, University of Oldenburg,
  Carl-von-Ossietzky-Str. 9-11, 26129 Oldenburg, Germany}
\email{keno.eilers@hotmail.de}

\author{V. Z. Enolski}
\address{School of Mathematics and Maxwell Institute, Edinburgh
  University, Edinburgh, EH9 3JZ, UK, on leave from the Institute of
  Magnetism, National Academy of Sciences of Ukraine, Kiev, 03142,
  Ukraine}
\email{Viktor.Enolskiy@ed.ac.uk}

\maketitle

\hfill{Dedicated to the 70th birthday of Victor Buchstaber}
\vskip 1cm
\begin{abstract}
  For elliptic curves, expressions for the periods of elliptic
  integrals of the second kind in terms of theta-constants, have been
  known since the middle of the 19th century.  In this paper we
  consider the problem of generalizing these results to curves of
  higher genera, in particular to a special class of algebraic curves,
  the so-called $(n,s)$-curves.  It is shown that the representations
  required can be obtained by the comparison of two equivalent
  expressions for the projective connection, one due to Fay-Wirtinger
  and the other from Klein-Weierstrass.  As a principle example, we
  consider the case of the genus two hyperelliptic curve, and a number
  of new Thomae and Rosenhain-type formulae are obtained.  We
  anticipate that our analysis for the genus two curve can be extended
  to higher genera hyperelliptic curves, as well as to other classes
  of $(n,s)$ non-hyperelliptic curves.

% The problem of generalisation of classical expressions for periods
%   of second kind elliptic integrals in terms of theta-constants to
%   higher genera is studied. In this context special class of algebraic
%   curves -- $(n,s)$-curves is considered.  It is shown that required
%   representations can be obtained by comparison of equivalent
%   expressions for projective connection by Fay-Wirtinger and
%   Klein-Weierstrass. The case of genus two hyperelliptic curve is
%   considered as a principle example and a number of new Thomae and
%   Rosenhain-type formulae are obtained. We anticipate that the
%   analysis undertaken for genus two curve can be extended to higher
%   genera hyperelliptic curve as well to other classes of $(n,s)$
%   non-hyperelliptic curves.
\end{abstract}
\section{Introduction}

We discuss the following problem, which is solved in particular
  cases: {\em Consider a curve $\mathcal{C}$ of genus $g>1$ and its
    $\mathfrak{a}$ and $\mathfrak{b}$-periods of holomorphic
    differentials $2\omega,2\omega'$.  Let $2\eta,2\eta'$ be the
    periods of the differentials of the second kind conjugated to
    $2\omega,2\omega'$ according to the generalized Legendre relations
\begin{equation}
\eta^T\omega = \omega^T\eta,\quad \eta^T\omega'-\omega^T \eta' =
\frac{\imath\pi}{2}, \quad {\eta'}^T\omega'={\omega'}^T\eta'.
\label{legendre}\end{equation}
Express the periods of the differentials of the second kind in terms
of the data, including $\mathfrak{a}$-periods, $2\omega$, and
$\theta$-constants, depending on the Riemann period matrix
$\tau=\omega'/\omega$, and coefficients of the polynomial defining the
curve $\mathcal{C}$}

In the case of the elliptic curve $y^2=4x^3-g_2x -g_3$, the question
posed  is answered by the Weierstrass formulae
\begin{align}
  \eta=-\frac{1}{12\omega} \sum_{k=2}^4
  \frac{\vartheta_k''(0)}{\vartheta_k(0)} \quad \text{and
    equivalently}\quad \eta=-\frac{1}{12\omega}
  \frac{\vartheta_1'''(0)}{\vartheta_1'(0)},
\label{Weierstrass1}
\end{align}
where $2\eta$ is the $\mathfrak{a}$-period of the elliptic
differential of the second kind $-x\mathrm{d}x/y$.  For algebraic
curves of higher genera, the {\it second period matrix},
$(2\eta,2\eta')$, in the terminology of \cite{mar06}, appears
naturally in the definition of the Riemann $\theta$-function; the
possibility of expressing it in terms of the {\it first period
  matrix}, $(2\omega,2\omega')$ and the $\theta$-constants is of
theoretical interest.  It is also useful for the implementation of
calculations by computer algebra, e.g.\ the current versions of
Maple/algcurves calculates only periods of differentials of the first
kind.  The $\theta$-constant representation of periods of
differentials of the second kind is also important for defining the
multi-dimensional $\sigma$-function which is currently of interest:
the basic theory of the $\sigma$-function has been intensively
developed over the last few decades by V.~M.~Buchstaber with
co-workers, see the recent manuscript \cite{bel12}.

Multi-variate $\sigma$-functions were introduced by F.\ Klein,
\cite{kle888}, who formulated the program of construction of Abelian
functions on the basis of these functions in \cite{kle890}. The
expression of periods of the second kind, in terms of
$\theta$-constants, also modular forms built on $\theta$-constants, is
a part of this program. Klein developed the theory in the cases of
hyperelliptic curves and arbitrary genus 3 curves to realise his
program in particular cases.  Later the hyperelliptic theory was well
documented by Baker \cite{bak897} and especially the case of the genus
two curve \cite{bak907}.  We present here an approach to solve the
problem formulated above for the family of $(n,s)$-curves; these
represent a natural generalization of the Weierstrass elliptic cubic
to higher genera.

$(n,s)$-curves are defined as follows. Let $n$ and $s$ be a pair of
co-prime integers such that $s>n \geq 2$. Those non-negative natural
numbers $w_1,w_2,\ldots,w_g$, which cannot be represented in the form
$\alpha n+\beta s$ with non-negative integer $\alpha$ and $\beta$,
form the Weierstrass gap sequence of length $g=(n-1)(s-1)/2$.  The
$(n,s)$-curve is the non-degenerate plane curve of genus $g$ given by
the polynomial equation
\begin{equation}
f(x,y)=y^n-x^s-\sum_{\alpha,\beta} \lambda_{\alpha n+\beta s} x^{\alpha}y^{\beta}=0
\label{nscurve}
\end{equation}
with $\lambda_k\in \mathbb{C}$ and $0\leq \alpha < s-1, \quad 0\leq
\beta < n-1$.

The notion of $(n,s)$-curves was introduced in \cite{bel999} and now
attracts much interest.  It is shown in \cite{bl08} that models of
$(n,s)$-curves are better adapted to the construction of
multi-variable $\sigma$-functions than models suggested by Weierstrass
and later models of mini-versal deformations of singularities of the
form $y^n=x^s$.  We also mention Klein \cite{kle888}, \cite{kle890}
and recent works presenting effective description of multi-variate
$\sigma$-functions \cite{nak10} and \cite{ksh12}.  In particular, the
approach of Klein (for any Riemann surface of genus 3) and
\cite{ksh12} (for an arbitrary Riemann surface of any genus) to the
theory of higher genus sigma-functions is based on resolving the
generalized Legendre relations in terms of theta-constants.  Using
$(n,s)$-curves, it is possible to develop the construction of Abelian
functions and the associated integrable PDEs in terms of the
$\sigma$-function of the trigonal curve, \cite{eemop08},
\cite{bego08}, to develop the study of space curves \cite{mat11},
\cite{an12}, to consider $\tau$-functions of integrable hierarchies as
$\sigma$-functions, \cite{nak10a}, \cite{he11}, to develop the
description of classical surfaces like Kummer, Coble surfaces
\cite{bel12}, \cite{egop13}, to describe Jacobi inversion on the
strata of non-hyperelliptic Jacobians \cite{mp08}, \cite{bef12}, to
develop number-theoretical problems \cite{beh05}, \cite{kmp12} and
others.

The present note aims to describe the moduli of the $\sigma$-function,
but we anticipate that its content and area of applicability should be
more general.

\section{The method}
Let $\mathcal{C}$ be a plane algebraic curve of genus $g$ given by the
polynomial equation $f(x,y)=0$. Introduce the bi-differential
$\Omega(Q,R)$ on $\mathcal{C}\times \mathcal{C}\ni (Q,R)$, which is
called the {\em canonical bi-differential of the second order} if it
is
\begin{itemize}

\item symmetric
\begin{equation}
  \Omega(Q,R)=\Omega(R,Q),\label{symmetric}
\end{equation}

\item normalized at $\mathfrak{a}$-periods:
\begin{equation}
  \oint_{\mathfrak{a}_k} \Omega(Q,R)=0, \quad k=1,\ldots,g \label{normalization}
\end{equation}

\item and has the only pole of the second order along the
diagonal, in other words it has the following expansion
\begin{equation}
  \Omega(Q,R) = \frac{{\mathrm d}\xi(Q){\mathrm d}\xi(R)}
  {(\xi(Q)-\xi(R))^2}+ \frac{1}{6}S(P)  + \text{higher order terms}  ,
  \label{expansion}
\end{equation}
where $\xi(Q)$ and $\xi(R)$ are local coordinates of points
$Q=(x,y)$ and $R=(z,w)$ in the vicinity of a point $P$ respectively.
The quantity $S(P)$ is called the {\it holomorphic
  projective connection}.
\end{itemize}
The bi-differential $\Omega(Q,R)$ is uniquely defined by the
conditions (\ref{symmetric}), (\ref{normalization}) and
(\ref{expansion}).  We realise it as follows. Let $\theta$ be the
standard $\theta$-function defined by its Fourier series
\begin{equation}
  \theta(\boldsymbol{z};\tau) =\sum_{\boldsymbol{n}\in \mathbb{Z}^g}
  \mathrm{exp}
  \left\{  \imath\pi \boldsymbol{n}^T\tau \boldsymbol{n}
  + 2\imath\pi \boldsymbol{z}^T\boldsymbol{n}    \right\},
\end{equation}
where $\tau$ is the Riemann period matrix of the normalised holomorphic
differentials $\boldsymbol{v}=(v_1,\ldots, v_g)^T$,
\begin{equation}
  \oint_{\mathfrak{a}_j} v_i=\delta_{i,j}, \quad\oint_{\mathfrak{b}_j} v_i=\tau_{i,j}
  , \quad i,j=1,\ldots,g.\label{holdiff1}
\end{equation}
Let $\boldsymbol{\mathfrak{A}}$ be a non-singular point of the
$\theta$-divisor, $\boldsymbol{\mathfrak{A}}\in (\theta)$,
i.e. $\theta(\boldsymbol{\mathfrak{A}})=0$, $\mathrm{grad} \,
\theta(\boldsymbol{\mathfrak{A}})\neq 0 $. Then
\begin{equation}\label{bidiffW}
  \Omega(Q,R)= \mathrm{d}_x \mathrm{d}_z \, \mathrm{ln}\, \theta
  \left(\boldsymbol{\mathfrak{A}}+\int_{R}^{Q}
    \boldsymbol{v};\tau\right),
  \quad Q=(x,y), R=(z,w).
\end{equation}

Such a realisation of $\Omega(Q,R)$ yields the following representation
of the holomorphic projective connection $S(P)$:

{\bf The Fay-Wirtinger representation} of $S(P)$ \cite{wir943} quoted
by Fay \cite{fay973}, p.~19, where we denote $S(P)=S_{FW}(P)$ in the
form
\begin{align}
  S_{FW}(P)=\left\{\int^{P}_{\cdot}H_{\mathfrak{A}},P\right\}(P)
  +\frac32\left( \frac{Q_{\mathfrak{A}}}{H_{\mathfrak{A}}}\right)^2(P)
  -2\frac{T_{\mathfrak{A}}}{H_{\mathfrak{A}}}(P),
\end{align}
where $\{\cdot,\cdot\}$ is the Schwartzian differential operator,
\begin{align}\begin{split}
  H_{\mathfrak{A}}(P)&=\sum_{i=1}^g\frac{\partial \theta} {\partial
    z_i} ( {\mathfrak{A}}) v_i(P), \\ Q_{\mathfrak{A}}(P) &=
  \sum_{i,j=1}^g
  \frac{\partial^2 \theta} {\partial z_i\partial z_j}
(\mathfrak{A}) v_i(P) v_j(P),\\
  T_{\mathfrak{A}}(P)& = \sum_{i,j,k=1}^g\frac{\partial^3 \theta}
  {\partial z_i\partial z_j\partial z_k} ({\mathfrak{A}}) v_i(P)
  v_j(P)v_k(P).\end{split} \label{HQT}
\end{align}
Here $\mathfrak{A}$ is a non-singular point of the $\theta$-divisor,
$\mathfrak{A}\in (\theta)$, $v_j, j=1,\ldots,g$, where the $v_j$ are
the normalized holomorphic differentials defined in (\ref{holdiff1}).

Our method is based on a comparison of the Fay-Wirtinger representation of
the projective connection $S(P)$, given above as $S_{FW}(P)$, and the
equivalent representation of Klein and Weierstrass, $S_{KW}(P)$,
described below.

{\bf The Klein-Weierstrass representation} of the projective
connection follows from the Klein-Weierstrass realisation of the
bi-differential $\Omega(Q,R)$ in algebraic form (see \cite{bak897},
\cite{bak907} and the recent review \cite{bel12}).  Namely, let $f(x,y)=0$
be the equation of the $(n,s)$-curve $\mathcal{C}$ of genus $g$ with
marked point $P_0=(\infty,\infty)$.  Then for arbitrary points
$(Q=(x,y),R=(z,w)) \in \mathcal{C} \times \mathcal{C}$, the
bi-differential $\Omega(Q,R)$ is represented in the form
\begin{align}
  \Omega(Q,R)=\frac{\mathcal{F}(Q,R)}{(x-z)^2}
  \frac{\mathrm{d}x}{f_y(x,y)}\frac{\mathrm{d}z}{f_w(z,w)}+
  2\boldsymbol{u}(Q)^T \varkappa \boldsymbol{u}(R). \label{KWOmega}
\end{align}
Here $\mathcal{F}(Q,R)$ is a polynomial of its variables and
$\boldsymbol{u}=(u_1,\ldots, u_g)^T$ is a vector of canonical
holomorphic differentials computed at points $Q$ and $R$.

In \cite{nak10}, Nakayashiki presented a description of the polynomial
${\mathcal F}(Q,R)$ as a polynomial of $(x,y;z,w)$ with the
coefficients in homogeneous polynomials (with respect to the weights)
of $\lambda_i$ (=coefficients of $f(x,y)$). Our approach is based on
the explicit algebraic expression of $\mathcal{F}(Q,R)$, whose
derivation is classically known and described as follows.

Let $w_1<w_2<\ldots<w_g$ be the Weierstrass gap sequence at the point
$P_0$ of length $g$.  Order the components of the vector
$\boldsymbol{u}$ of holomorphic differentials in such a way that the
orders of vanishing at the point $P_0$ of the holomorphic integrals
are $\mathrm{ord} \left( \int u_k\right)=w_{g-k+1}$.  Represent the
basis of canonical holomorphic differentials in the form
\begin{equation}
  \boldsymbol{u}(x,y) = \frac{\mathrm{d}x}{f_y(x,y)}
  \boldsymbol{\mathcal{U}}(x,y),\label{holldiff}
\end{equation}
with vector $\boldsymbol{\mathcal{U}}(x,y)=(\mathcal{U}_1(x,y),\ldots,
\mathcal{U}_g(x,y))^T$, whose components are monomials constructed by
the Weierstrass gap sequence.  Introduce the vector
$\boldsymbol{\mathcal{R}}(x,y)=(\mathcal{R}_1(x,y), \ldots,
\mathcal{R}_g(x,y))^T$ defining the conjugate meromorphic differential
with the only pole at the point $P_0$.
\begin{equation}
  \boldsymbol{r}(x,y)= \frac{\mathrm{d}x}{f_y(x,y)}
  \boldsymbol{\mathcal{R}}(x,y).
\end{equation}
The vector $\boldsymbol{\mathcal{R}}(x,y)$ is constructed in such a
way that the period matrices
\begin{align*}
2\omega & = \left(\oint_{\mathfrak{a}_j}u_i\right)_{i,j=1,\ldots,g},
\quad 2\omega' = \left(\oint_{\mathfrak{b}_j}u_i\right)_{i,j=1,\ldots,g},\\
2\eta &= -\left(\oint_{\mathfrak{a}_j}r_i\right)_{i,j=1,\ldots,g},
\quad 2\eta' = -\left(\oint_{\mathfrak{b}_j}r_i\right)_{i,j=1,\ldots,g}
\end{align*}
satisfy the generalized Legendre relation
\begin{equation}
\left(\begin{array}{cc}
\omega&\omega'\\ \eta& \eta'
\end{array}\right)
\left(  \begin{array} {cc} 0_g&-1_g\\
    1_g&0_g\end{array}\right)  \left(\begin{array}{cc}
    \omega&\omega'\\
    \eta& \eta'  \end{array}\right)^T=-\frac{\imath \pi}{2}
\left(  \begin{array} {cc} 0_g&-1_g\\
    1_g&0_g\end{array}\right), \label{genlegendre}
\end{equation}
which was given in expanded form in (\ref{legendre}).

The polynomial ${\mathcal F}(Q,R)={\mathcal F}(x,y;z,w)$ and vector
$\boldsymbol{\mathcal{R}}(x,y)$ are found simultaneously within the
following construction.  Introduce vectors $\boldsymbol{\phi}(x,y)$
and $\boldsymbol{\psi}(x,y)$
\begin{align}
\boldsymbol{\phi}(x,y)=(y^{n-1},y^{n-2},\ldots,1)^T
\end{align}
and
\begin{align}
\boldsymbol{\psi}(x,y)=(1,\psi_1(x,y),\ldots,\psi_{n-1}(x,y))^T,
\end{align}
where
\[
\psi_k(x,y) =\left( \frac{f(x,y)}{y^{n-k}} \right)_+,  \quad
k=1,\ldots n-1
\]
and the subscript $+$ means that, after division, only monomials of
non-negative power in $y$ are taken into account in the sum.

Introduce the differential of the third kind $\Pi_{P_1,P_2}(P)$,
$P=(x,y)$ with first order poles at two points $P=P_1=(x_1,y_1)$,
$P_2=(x_2,y_2)$, and residues $\pm1$.  It is given, in terms of the
quantities introduced above, as
\begin{align}
\Pi_{P_1,P_2}(P)=\frac{\mathrm{d} x}{f_y(x,y)}\left\{
  \frac{\boldsymbol{\psi}^T(x_1,y_1) \boldsymbol{\phi}(x,y)}{x-x_1}   -
\frac{\boldsymbol{\psi}^T(x_2,y_2) \boldsymbol{\phi}(x,y)}{x-x_2}
 \right\}.
\end{align}
To construct the bi-differential $\Omega(Q,R)$ explicitly, we consider
on $\mathcal{C}\times \mathcal{C}$ the auxiliary bi-differential
\begin{equation}
  \widetilde{\Omega}(Q,R) = \mathrm{d}z\frac{\mathrm{d} x} {f_y(x,y)}
  \frac{\partial}{\partial z} \frac{\boldsymbol{\psi}^T(z,w)
    \boldsymbol{\phi}(x,y)}{x-z}, \quad Q=(x,y),\, R=(z,w). \label{auxillary}
\end{equation}
It is straightforward to show that
\[
\boldsymbol{\phi}(x,y)^T \boldsymbol{\psi}(x,Y)=
\frac{f(x,Y)-f(x,y)}{Y-y}   ,
\]
therefore $\widetilde{\Omega}(Q,R)$ has a pole of the second order
along the diagonal as a form in $x$.  Since this differential is
holomorphic in $(x,y)$ away from the diagonal, it has poles in the
variable $(z,w)$ at $z=\infty$.  Restore the symmetry by setting
\[
\widetilde{\widetilde{\Omega}}(Q,R)=\widetilde{\Omega}(Q,R)
+\boldsymbol{\mathcal{R}}(z,w)^T \boldsymbol{\mathcal{U}}(x,y)
\frac{\mathrm{d}x\mathrm{d}z}{f_y(x,y) f_w(z,w)},
\]
where the $g$-vector $\boldsymbol{\mathcal{U}}(x,y)$ is defined by the
holomorphic differentials (\ref{holldiff}), and the $g$-vector
$\boldsymbol{\mathcal{R}}(z,w)$ is found from the condition
\begin{align}\begin{split}
    &\frac{\partial f(z,w)}{\partial w} \frac{\partial }{\partial z}
    \frac{\boldsymbol{\psi}(z,w)^T \boldsymbol{\phi}(x,y)}{x-z}
    -\frac{\partial f(x,y)}{\partial y} \frac{\partial }{\partial x}
   \frac{\boldsymbol{\psi}(x,y)^T \boldsymbol{\phi}(z,w)}{z-x}\\
    &\hskip2cm
    =\boldsymbol{\mathcal{R}}^T(x,y)\boldsymbol{\mathcal{U}}(z,w)-
    \boldsymbol{\mathcal{R}}^T(z,w)\boldsymbol{\mathcal{U}}(x,y).
\end{split}\label{symmetry1}
\end{align}

As a result, the polynomial ${\mathcal F}(x,y;z,w)$ is found from the
relation
\begin{align*}
  \frac{\mathrm{d} x\mathrm{d}z} {f_y(x,y)}\frac{\partial}{\partial z}
  \frac{\boldsymbol{\psi}^T(z,w) \boldsymbol{\phi}(x,y)}{x-z}
 & +\boldsymbol{\mathcal{R}}(z,w)^T \boldsymbol{\mathcal{U}}(x,y)
  \frac{\mathrm{d}x\mathrm{d}z}{f_y(x,y)
    f_w(z,w)}\\ & =\frac{\mathcal{F}(x,y;z,w)}{(x-z)^2}
  \frac{\mathrm{d}x}{f_y(x,y)}\frac{\mathrm{d}z}{f_w(z,w)},
\end{align*}
and the bi-differential $\Omega(Q,R)$ is represented in the form
\begin{align}
\begin{split}
  \Omega(Q,R)&=\widetilde{\Omega}(Q,R) +
  \boldsymbol{\mathcal{R}}(z,w)^T \boldsymbol{\mathcal{U}}(x,y)
  \frac{\mathrm{d}x\mathrm{d}z}{f_y(x,y) f_w(z,w)}\\ &
  +2\boldsymbol{u}^T(Q)\varkappa\boldsymbol{u}(R).
\end{split}\label{omega3}
\end{align}

Explicit formulae for the polynomials $\mathcal{R}_j(x,y)$ are
classically known for hyperelliptic curves \cite{bak897} and found
only recently for the family of $(3,s)$ - trigonal curves (see review
\cite{bel12} where all necessary references are given).  We note that
the polynomials $\mathcal{R}_j(x,y)$ are not uniquely defined, and an
arbitrary polynomial built in $\mathcal{U}_j(x,y)$ can be added
without affecting the generalized Legendre condition
(\ref{genlegendre}).

The normalizing matrix $\varkappa$ is given according to the
construction as the symmetric matrix
\begin{equation}
  \varkappa^T=\varkappa, \quad \text{and} \quad
\varkappa= \eta (2\omega)^{-1}.
\end{equation}
The period matrices $2\eta$ and $2\eta'$ are expressible in terms of
$\varkappa$ and $\omega$, $\omega'$:
\begin{equation}
  \eta = 2\varkappa \omega, \qquad  \eta'=2 \varkappa \omega'
  -\frac{\imath \pi}{2}(\omega^{-1})^T\label{etaeta}.
\end{equation}

Now we are in a position to define the fundamental $\sigma$-function
of the $(n,s)$ curve $\mathcal{C}$ of genus $g=(n-1)(s-1)/2$.  To do
that we introduce the $\theta$-function
$\theta[\varepsilon](\boldsymbol{z};\tau)$ with characteristic
\[
[\varepsilon]= \left[   \begin{array}{c}  \boldsymbol{\varepsilon}^T\\
    \boldsymbol{\varepsilon'}^T\end{array} \right] =
\left[\begin{array}{ccc}
    \varepsilon_1&\ldots&\varepsilon_g \\
    \varepsilon_1'&\ldots&\varepsilon_g'
\end{array} \right]
\]
with half-integer $\varepsilon_i,\varepsilon_j'= 0$ or $1/2$ as the
Fourier series
\begin{align}\begin{split}
  \theta[\varepsilon](\boldsymbol{z};\tau)
  &=\mathrm{exp}\left\{  \imath\pi \boldsymbol{\varepsilon}^T\tau
\boldsymbol{\varepsilon} + 2\boldsymbol{\varepsilon}^T(\boldsymbol{z} +
  \boldsymbol{\varepsilon}') \right\}
\theta(\boldsymbol{z+\tau \boldsymbol{\varepsilon}
+\boldsymbol{\varepsilon}'})\\
  &=\sum_{\boldsymbol{n}\in
    \mathbb{Z}^g}\mathrm{exp} \left\{
    \imath\pi(\boldsymbol{n}+\boldsymbol{\varepsilon})^T\tau
    (\boldsymbol{n}+\boldsymbol{\varepsilon}) +2\imath\pi
    (\boldsymbol{n}+\boldsymbol{\varepsilon})^T
    (\boldsymbol{z}+\boldsymbol{\varepsilon'})\right\}
\end{split} \label{thetachar}
\end{align}
Here the Riemann matrix $\tau=\omega'/\omega$ necessarily belongs to
the Siegel upper-half space, i.e. $\tau^T=\tau$ and
$\mathrm{Im}\tau>0$. The $\theta$-function with characteristic
$[\varepsilon]$ is an even or odd function of $\boldsymbol{z}$
whenever $4\boldsymbol{\varepsilon}^T \boldsymbol{\varepsilon'}=0 \;
\mathrm{mod}\; 2$ or $4\boldsymbol{\varepsilon}^T
\boldsymbol{\varepsilon'}=1 \; \mathrm{mod}\; 2$.

The vector  of Riemann constants $\boldsymbol{K}_{P_0}$   with the
base point at $P_0=(\infty,\infty)$ is defined by the condition
\begin{equation}
\theta(\boldsymbol{\mathcal{A}}(D)+\boldsymbol{K}_{P_0})\equiv 0
\end{equation}
for all divisors linearly equivalent to divisors $D$ of degree $g-1$,
$D=P_1+\ldots+P_{g-1}$, where $\boldsymbol{\mathcal{A}}(D)$ is the
Abel map with base point $P_0$,
\[
\boldsymbol{\mathcal{A}}(D) =\sum_{k=1}^{g-1} \int_{P_0}^{P_k}
\boldsymbol{v}, \qquad \boldsymbol{v}=(2\omega)^{-1}\boldsymbol{u} .
\]
The vector of Riemann constants with base point $P_0$ is a half-period
if and only if there exists a holomorphic differential with zero of
order $2g-2$ at the point $P_0$, see e.g. \cite{fk980}, p.~299.
Remarkably, the $(n,s)$-curve admits such a differential \cite{nak10},
namely the differential $u_1=\mathrm{d}x/f_y(x,y)$ that has a zero of
required order in the point $P_0=(\infty,\infty)$.  Indeed,
\[
\left.\mathrm{d}x/f_y(x,y)\right|_{x=1/\xi^n,y=1/\xi^s} \sim
\xi^{ns-n-s-1}\mathrm{d}\xi=\xi^{2g-2}\mathrm{d}\xi,
\]
where we use the expression for the genus of the $(n,s)$-curve,
$g=(n-1)(s-1)/2$. Therefore, the divisor $(g-1)P_0$ defines a spin
structure, and the corresponding half-integer $\theta$-characteristic
$\gamma$ defines a $\sigma$-function, which is called the {\it
  fundamental $\sigma$-function of the $(n,s)$ curve}
\begin{align}
  \sigma(\boldsymbol{z})=C \theta[\gamma]
  ((2\omega)^{-1} \boldsymbol{z};\tau )\mathrm{exp}
  \left\{\boldsymbol{z}^T\varkappa
    \boldsymbol{z}\right\},\label{sigmafund}
\end{align}
where $\boldsymbol{z}$ is a vector from the Jacobi variety
$\mathrm{Jac}(\mathcal{C})=\mathbb{C}^g / 2\omega\otimes 2\omega'$,
$C$ is a constant independent of the variable $\boldsymbol{z}$; for
details and other definition of $\sigma(\boldsymbol{z})$ see, e.g.\
\cite{bel12}.  According to Klein \cite{kle890}, the quadratic form in
the exponential $\sum_{i,j} \varkappa_{i,j} z_i z_j$ should be
presented in terms of $\theta$-constants; such an expression in terms
of logarithmic derivatives of even $\theta$-constants, $\partial
\mathrm{ln}\, \theta[\varepsilon]/ \partial \tau_{i,j}$, was given in
\cite{kle890}.

We remark that the characteristic $[\gamma]$ can be odd and even, and
the corresponding $\sigma$-function inherits this parity as a function
of $\boldsymbol{z}$.  This parity coincides with parity of the integer
number $(n^2-1)(s^2-1)/24$, see \cite{bel999a}.  This last number
represents the order of vanishing of $\sigma(\boldsymbol{z})$ in
$\mathrm{Jac}(\mathcal{C})$ over the variable $\xi$, when components
of the vector $\boldsymbol{z}$ are graded as $z_k=\xi^{w_{g-k+1}}$,
and $(w_1,\ldots,w_g)$ are the Weierstrass gap numbers at the point
$P_0=(\infty,\infty)$.

It is also useful to mention that, since $(2g-2)P_0$ belongs to the
canonical class, it also follows \cite{ksh12} that the dimension of
the moduli space of algebraic curves that can be represented as
$(n,s)$-curves is equal to $2g-1$, that is the variety of such curves
has co-dimension $g-2$ in the space of moduli of all curves. In
\cite{ksh12}, the theory is constructed for arbitrary curves, whilst
we concentrate here on the moduli problem of $\sigma$-functions of
$(n,s)$-curves, finding various representations of the matrix
$\varkappa$ in terms of $\theta$-constants.  The following result is
important for the development of our argument.

\begin{prop}
Let the $(n,s)$ curve be given in the form
\begin{equation}
y^n-a_{n-1}(x)y^{n-1}-\ldots - a_1(x)y - a_{0}(x)=0.\label{nscurve2}
\end{equation}
The projective connection $S_{KW}(x,y)$ at the point $P=(x,y)\in
  \mathcal{C}$ and with the local coordinate $\xi(P)$ is given as
\begin{align}\begin{split}
    S_{KW}(x,y) = & \left\{x,\xi\right\}\mathrm{d}\xi^2+\mathcal{T}(x,y)\\
    & +6\boldsymbol{r}^T(x,y) \boldsymbol{u}(x,y)
    +12\boldsymbol{u}^T(x,y)\varkappa\boldsymbol{u}(x,y),
\end{split}
\label{SHGEN}
\end{align}
where
\begin{equation}
\mathcal{T}(x,y)=-\frac{1}{2f_y}\left[3y''f_{yy}+2{y'}^2f_{yyy} +6y'f_{yyx} +
      6f_{yxx}  \right]\mathrm{d}x^2.\label{tcal}
\end{equation}
Here the subscripts mean partial differentiation, the primes ${}'$
mean implicit differentiation of $y$ as a function of $x$ given by the
equation of the curve, $f(x,y)=0$ and $\{\cdot,\cdot\}$ is the
Schwartzian derivative,
\[
\{x(\xi),\xi\} = \displaystyle{\frac{\mathrm{d}^3 x(\xi)
    /\mathrm{d}\xi^3}
{\mathrm{d} x(\xi) /\mathrm{d}\xi} } -
\frac32\left (\displaystyle{\frac{\mathrm{d}^2 x(\xi) /
      \mathrm{d}\xi^2} { \mathrm{d} x(\xi) /\mathrm{d}\xi}
  }\right)^2.
\]
\end{prop}

\begin{proof} Currently we have only a proof by direct calculation for
  $(n,s)$ curves up to and including $n=10$, a more general proof is
  under investigation.  Consider the representation (\ref{omega3}) of
  the bi-differential $\Omega(Q,R)$.  The last two terms in
  (\ref{omega3}), after restriction to the diagonal $Q=R$, produce the
  last two terms (after multiplication by 6) in (\ref{SHGEN}).  The
  expansion of $\widetilde{\Omega}(Q,R) $ is of the form
$$
\widetilde{\Omega}(Q,R)=\frac{\mathrm{d}\xi(Q) \mathrm{d}{\xi}(R)}{ (
  \xi(Q)-\xi(R) )^2} + \{ x(\xi),\xi \}
\mathrm{d}\xi(Q)\mathrm{d}\xi(R) +\mathcal{T}(Q,R),
$$
where the quantity $\mathcal{T}(Q,R)$, when restricted to the
diagonal, $Q=R=P$, should be shown to be of the form (\ref{tcal}).
This can be done by direct calculation in the following way. Consider
successive cases $n=2,3,\ldots$ fixing the $(n,s)$-curve in the form
(\ref{nscurve2}).  We get:
\begin{align*}
  n=2:&\quad f(x,y)=y^2-a_1(x)y-a_0(x)=0,\\
&  \mathcal{T}=- \frac{3}{f_y} (y''-a_1'')  \mathrm{d} x^2,\\
  n=3:&\quad f(x,y)=y^3-a_2(x)y^2-a_1(x)y-a_0(x)=0, \\
&\mathcal{T}=-\frac{3}{f_y} \left\{ (3y-a_2)y''+2{y'}^2-2a_2'y'-a_1''
  \right\} \mathrm{d} x^2,\\
  n=4:&\quad   f(x,y)=y^4-a_3(x)y^3-a_2(x)y^2-a_1(x)y-a_0(x)=0,\\
&\mathcal{T}=-\frac{3}{f_y}
  \left\{(6y^2-3ya_3-a_2)y''+(8y-2a_3){y'}^2-(6ya_3'+2a_1')y'\right.\\
&\hskip1.5cm\left.-3y^2a_3'' -2ya_2''-a_1''\right\} \mathrm{d} x^2,\\
  n=5:&\quad f(x,y)=y^5-\sum_{k=0}^4 a_k(x)y^k=0,
\\& \mathcal{T}=-\frac{3}{f_y} \left\{(10y^3-6a_4y^2-3a_3y-a_2)y''
+(20y^2-8a_4y-2a_3){y'}^2\right. \\
  &\hskip1.5cm \left.   -(12a_4y^2+6ya_3'+2a_2')y'-4a_4''y^3
    -3y^2a_3''-2ya_2''-a_1''\right\}
\mathrm{d} x^2. \\
  &\qquad \vdots \qquad \qquad \qquad \vdots\qquad \qquad \qquad
  \vdots\qquad \qquad \qquad \vdots\\
\end{align*}
Analysing these formulae, we find the general representation of the
term $\mathcal{T}(P)$ given in the statement of the proposition. In
the general case this result is currently just a conjecture, and a
general proof is in the process of development.
\end{proof}

The procedure for the derivation of relations for periods of the
differentials of the second kind is based on the uniqueness of the
normalized bi-differential $\Omega(Q,R)$ and is described as follows.
Let $\mathcal{C}$ be a $(n,s)$-curve.  Let $P$ be a point in the
vicinity of $P_0=(\infty,\infty)$ where the local coordinate $\xi(P)$
is given as $x=1/\xi^{n}$.  Let $\mathfrak{A}$ be a non-singular point
of the $\theta$-divisor, $\mathfrak{A}\in (\theta)$, satisfying
the condition
\begin{equation}
  H_{\boldsymbol{\mathfrak{A}}} (P_0) \neq 0.\label{condition}
\end{equation}
Expanding both parts of the equivalence
\begin{equation}
  S_{FW}(P) \equiv S_{KW}(P)\label{master}
\end{equation}
in $\xi(P)$, and equating coefficients of corresponding powers, we
obtain relations involving matrix elements of $\varkappa$ and
$\theta$-derivatives at the point $\mathfrak{A}$.  In this way we
obtain a compatible system of linear equations for the
$\varkappa_{i,j}$.

{\bf Remark} The Weierstrass formula (\ref{Weierstrass1}) can be
easily obtained by expanding both sides of the Weierstrass
representation of the $\sigma$-function in terms of the Jacobian
$\theta$-function using the Weierstrass series for $\sigma(u)$ with
coefficients given recursively.  This method can be generalized to
higher genera curves in the cases where $\sigma$-expansions are
known.  But only isolated cases of such expansions are elaborated -
the Buchstaber-Leykin recursion for genus two $\sigma$-functions
\cite{bl05}, calculation of first few terms of $\sigma$-expansions of
$(3,4)$-curve presented in \cite{eemop08}, $(3,5)$ \cite{bego08},
$(3,7)$ and $(3,8)$-curves are given in \cite{eng09} and $(4,5)$-curve
in \cite{ee09}.  We see an advantage of the method proposed, since
the derivation is reduced to examining of series over one complex
variable, the local coordinate, whilst the generalization of the
Weierstrass method leads to a series in many variables.

In the next section, we will demonstrate how the approach works in the
case of the hyperelliptic curve.

\section{Hyperelliptic curve}
We realise the hyperelliptic curve $\mathcal{C}$ of genus $g$ in the
form
\begin{equation}
  y^2=4x^{2g+1}+\lambda_{2g}x^{2g} \lambda_{2g-1}\lambda_{2g} +\ldots
+ \lambda_0,\quad \lambda_i\in\mathbb{C}.    \label{hypcurve}
\end{equation}
The basis of holomorphic and meromorphic differentials can be fixed as
\begin{align}\begin{split}
    u_i&= \frac{x^{i-1}}{y}\mathrm{d}x,\quad i=1,\ldots,g,\\
    r_j&=\sum_{k=j}^{2g+1-j}(k+1-j)\lambda_{k+1+j}
    \frac{x^{k}\mathrm{d}x}{4y}, \quad j=1,\ldots,g.
\end{split} \label{bakerbasis}
\end{align}
The matrices of their $\mathfrak{a}$ and $\mathfrak{b}$-periods,
$2\omega,2\eta$ and $2\omega',2\eta'$ satisfy the generalized Legendre
relation (\ref{genlegendre}).  As we already noted, the meromorphic
differentials $r_i$ are not uniquely determined and a linear
combination of holomorphic differentials can be added without
violation of the relation (\ref{genlegendre}).  This basis in the
space of meromorphic differentials was introduced by Baker
\cite{bak897} and we will call it the {\it Baker basis.}

The normalized canonical bi-differential has the form
\begin{align}\begin{split}
    \Omega (x,y; z,w) = \frac{2y w+F (x,z) }{ 4 (x-z) ^2} \frac{
      \mathrm{ d}x}{ y} \frac{ \mathrm{ d}z}{w} +2 \sum_{i,j=1}^g
    x^{i-1}z^{j-1}\varkappa_{i,j} \frac{ \mathrm{ d}x}{ y} \frac{
      \mathrm{ d}z}{w}.
  \end{split} \label{omega1}
\end{align}
Here $F(x,z)$ is the {\it Kleinian 2-polar}, with $F(x,x)=2y^2$, given by
the formula
\begin{align}
  F(x,z)= \sum_{k=0}^g x^kz^k(2 \lambda_{2k}+ \lambda_{2k+1}(x+z) ).
\end{align}
The symmetric matrix
\begin{align}
  \varkappa =\left( \begin{array}{ccc}
      \varkappa_{1,1}& \ldots& \varkappa_{1,g}  \\
      \vdots&&\vdots\\
      \varkappa_{1,g}& \ldots& \varkappa_{g,g}
\end{array}\right)= \eta    (2\omega)^{-1},
\end{align}
and therefore the period matrices of differentials of the second kind
$2\eta,2\eta'$, are computed by the known $\varkappa$ as in
(\ref{etaeta}).

When the bi-differential is found in the form (\ref{omega1}), then the
fundamental $\sigma$-function can be defined as in (\ref{sigmafund}),
and the Klein-Weierstrass $\wp$-functions introduced,
\begin{align}\begin{split}
    &\wp_{ij}(\boldsymbol{z})=-\frac{\partial^2}{\partial z_i\partial
      z_j}
    \, \mathrm{ln}\, \sigma(\boldsymbol{z}), \quad i,j, = 1,\ldots, g,\\
    &\wp_{ijk}(\boldsymbol{z})=-\frac{\partial^3}{\partial
      z_i\partial z_j \partial z_k} \, \mathrm{ln}\,
    \sigma(\boldsymbol{z}), \quad i,j, k = 1,\ldots, g, \quad \text{
      etc}. \end{split} \label{wp}
\end{align}

The solution of the the Jacobi inversion problem found by Baker (see
\cite{bel12}) looks remarkable simple in the $\wp$-variables:

\begin{align}\begin{split}
    &x^g-\wp_{gg}(\boldsymbol{z}) x^{g-1}
    -\wp_{g,g-1}(\boldsymbol{z})
    x^{g-2} -\ldots - \wp_{1g}(\boldsymbol{z})=0,\\
    &y_k=\wp_{ggg}(\boldsymbol{z})x_k^{g-1}
    +\wp_{g-1,gg}(\boldsymbol{z}) x_k^{g-2}+ \ldots
    +\wp_{2,gg}(\boldsymbol{z})x_k +\wp_{1,gg}(\boldsymbol{z}),\\
    &\qquad \qquad k=1,\ldots,g.
\end{split} \label{JIP}
\end{align}

\begin{prop}
  Let $\xi(P)$ be the local coordinate of the point $P$.  The
  projective connection is given by the formula
\begin{align}
\begin{split}
  S_{KW}(x,y)&=\{x(\xi),\xi\}\mathrm{d}\xi^2
  -\frac32\frac{y''(x)}{y(x)}\mathrm{d}x^2\\
  & +6\boldsymbol{u}^T(x,y)\boldsymbol{r}(x,y)
  +12\boldsymbol{u}^T(x,y)\varkappa\boldsymbol{u}(x,y).
\end{split}\label{SH}
\end{align}
\end{prop}
\begin{proof}
  The formula (\ref{SH}) represents a particular case of (\ref{SHGEN})
  or can be derived in an analogous way.
\end{proof}

Consider some further special cases.

\subsection{Elliptic curve} As already mentioned, the formula for
$\varkappa=\eta/2\omega$ in the case of elliptic curve can be obtained
by the expansion of $\sigma(u)$.  But we will demonstrate how the method
works for the elliptic curve
\begin{equation}
  y^2=4x^3+\lambda_2x^2+\lambda_1x+\lambda_0
\end{equation}
and the infinite point $P_0$, where the local coordinate of a point
$P$ is introduced as $x=1/\xi^2\sim \infty$.
\begin{align*}
  S_{KW}(P)/(\mathrm{d}\xi)^2\sim -\frac34\lambda_2+12 \varkappa+\left(
    -\frac32\lambda_1+\frac{9}{32}\lambda_2^2-3\varkappa \lambda_2
  \right)\xi^2+O(\xi^4).
\end{align*}
To expand $S_{FW}(P)$, we choose $\mathfrak{A}=1/2+\tau/2 \in (\theta)$
and write the expression for $S_{FW}$ in terms of Jacobian
$\theta$-functions. We have for the quantities (\ref{HQT})
\[
H_{\mathfrak{A}}(P) = -\frac{\mathrm{d}\xi}{2\omega}\vartheta'_1(0)
\mathrm{e}^{-\imath\pi\tau/4}, \quad Q_{\mathfrak{A}}(P)=0,\quad
T_{\mathfrak{A}}(P) =
-\frac{\mathrm{d}\xi}{8\omega^3}\vartheta'''_1(0)
\mathrm{e}^{-\imath\pi\tau/4} .
\]
Taking this into account, we write
\begin{align*}
  S_{FW}/(\mathrm{d}\xi)^2\sim -\frac14\lambda_2 -\frac{1}{2\omega^2}
  \frac{\vartheta_1'''(0)} {\vartheta_1'(0)}
  +\left(\frac{5}{32}\lambda_2^2 -\frac32\lambda_1 +\frac{\lambda_2}
    {8\omega^2}\frac{\vartheta_1'''(0)}{\vartheta_1'(0)}\right)\xi^2+O(\xi^4)
\end{align*}
Equating coefficients of expansions in (\ref{master}), we get
\begin{equation}
\varkappa=\frac{\eta}{2\omega} = \frac{1}{24}\lambda_2
-\frac{1}{24\omega^2} \frac{\vartheta_1'''(0)}{\vartheta_1'(0)}.
\label{Weierstrass3}
\end{equation}
At $\lambda_2=0$, the Weierstrass formula (\ref{Weierstrass1})
follows. We can check that we get again (\ref{Weierstrass3}) on
equating coefficients of the $\xi^2$ term.  In the next section, we
develop the same expansion procedure for the genus two curve, omitting
details of the underlying computer algebra calculations.

\subsection{Genus two curve}
As the principal example in this paper, consider a genus two
hyperelliptic curve of the form (\ref{hypcurve})
\begin{align}\begin{split}
    y^2&=4x^5+\lambda_4x^4+\lambda_3x^3+\lambda_2x^2 +\lambda_1x +\lambda_0\\
    &=4(x-e_1)(x-e_2)(x-e_3)(x-e_4)(x-e_5),
\end{split}\label{g2curve}
\end{align}
where $e_1,\ldots, e_5$ are finite branch points.

Fix a homology basis $\mathfrak{a}_1 ,\mathfrak{a}_2; \mathfrak{b}_1,
\mathfrak{b}_2$ and calculate the $2\times2$ period matrices of the
holomorphic and meromorphic differentials $2\omega, 2\omega',
2\eta,2\eta'$.  Also introduce the Riemann period matrix $\tau$ and
the necessarily symmetric matrix $\varkappa$,
\[
\tau= \omega^{-1}\omega' , \quad \tau^T =\tau,\, \mathrm{Im}\, \tau>0,
\quad\text{and}\quad \varkappa = \frac{1}{2} \eta \omega^{-1}, \; \varkappa^T =
\varkappa.
\]
Denote
\begin{equation}
(2\omega)^{-1} = (\boldsymbol{U},\boldsymbol{V}),
\end{equation}
where $\boldsymbol{U}=(U_1,U_2)^T$ and $\boldsymbol{V}=(V_1,V_2)^T$
are 2-column vectors called {\it winding vectors}.  We will use below the
directional derivatives denoted by
\begin{align*}
  \partial_{\boldsymbol{U}} \Psi(\boldsymbol{z}) &=U_1
  \frac{\partial}{\partial z_1}\Psi(z_1,z_2)+U_2
  \frac{\partial}{\partial z_2}\Psi(z_1,z_2),
  \\ \partial_{\boldsymbol{V}} \Psi(\boldsymbol{z}) &=V_1
  \frac{\partial}{\partial z_1}\Psi(z_1,z_2)+V_2
  \frac{\partial}{\partial z_2}\Psi(z_1,z_2),
\end{align*}
where $ \Psi(\boldsymbol{z})$ is a function of two complex variables.
Higher derivatives
$\partial_{\boldsymbol{U}^2}, \partial_{\boldsymbol{U}\boldsymbol{V}}$
are defined accordingly.

Beside the canonical $\theta$-function, $\theta$-functions with
half-integer characteristics will be used to formulate our
result
\begin{align}
\begin{split}
  \theta[\varepsilon](\boldsymbol{z};\tau) & = \theta
  \left[  \begin{array}{cc}  \boldsymbol{ \varepsilon}^T \\
      \boldsymbol{\varepsilon'}^T \end{array}\right](\boldsymbol{z};\tau)
  = \theta \left[  \begin{array}{cc}   \varepsilon_1& \varepsilon_2\\
      \varepsilon_1'&\varepsilon_2'  \end{array}\right](z_1,z_2;\tau) \\
  &=\sum_{\boldsymbol{n}\in \mathbb{Z}^2} \mathrm{exp} \left\{
    \imath\pi \left(
      \boldsymbol{n}+\boldsymbol{\varepsilon}\right)^T\tau \left(
      \boldsymbol{n}+\boldsymbol{\varepsilon}\right)+2\imath\pi \left(
      \boldsymbol{z}+\boldsymbol{\varepsilon}'\right)^T\left(
      \boldsymbol{n}+\boldsymbol{\varepsilon}\right) \right\}
\end{split} \label{theta}
\end{align}
with characteristics $\varepsilon_i,\varepsilon_i'=0 $ or $1/2$. We
will also use the $\theta$-relation for characteristics
$[\varepsilon]$ and $[\rho]$
\begin{align}
\begin{split}
  &\theta[\varepsilon](\boldsymbol{z}+\tau\boldsymbol{ \rho} +
  \boldsymbol{\rho}';\tau)\\&=
  \theta[\varepsilon+\rho](\boldsymbol{z};\tau) \mathrm{exp}\left\{
    -\imath\pi\boldsymbol{\rho}^T\tau\boldsymbol{\rho} -2\imath\pi
    \boldsymbol{\rho}^T\boldsymbol{z} -2\imath\pi(\boldsymbol{\rho}'
    +\boldsymbol{\varepsilon}')^T\boldsymbol{\rho} \right\}.
\label{thetarelation}
\end{split}
\end{align}

There are sixteen $\theta$-functions with characteristics, 6 of which
are odd functions of $\boldsymbol{z}$, and 10 of which are even.  We
denote the corresponding odd (even) characteristics as $[\delta_1],
\ldots [\delta_6]$ ($[\varepsilon_1],\ldots, [\varepsilon_{10}]$). We
will also call the half-period
$\tau\boldsymbol{\varepsilon} +\boldsymbol{\varepsilon}'$ odd or even
whenever its characteristic is odd or even.

The $\theta$-functions with characteristics and zero argument are
called $\theta$-constants.  There are 10 non-vanishing even
$\theta$-constants $\theta[\varepsilon_i]
=\theta[\varepsilon_i](\boldsymbol{0};\tau)$, $i=1,\ldots,10$.  The 6
derivative odd $\theta$-constants satisfy
\[
\begin{split}
\theta_1[\delta_i]
= \partial\theta[\delta_i](z_1,z_2;\tau)/\partial z_1
\vert_{\boldsymbol{z}=0},\quad \theta_2[\delta_i]
 = \partial\theta[\delta_i](z_1,z_2;\tau)/ \partial
z_2 \vert_{\boldsymbol{z}=0},
\end{split}
\]
$ \theta_1[\delta_i],\theta_2[\delta_i] \neq 0, $ simultaneously for
all $i=1,\ldots,6$, i.e. $\mathrm{grad}\, \theta[\delta_i]\neq 0$.

The following half-period is odd
\[
\mathfrak{A}_i= \tau \left(  \begin{array}{c}  \delta_{i,1}\\
    \delta_{i,2} \end{array} \right)+\left( \begin{array}{c}
    \delta_{i,1}'\\ \delta_{i,2}' \end{array} \right), \quad
\mathfrak{A}_i\in (\theta)
\]
and belongs to the $\theta$-divisor, $(\theta)$.  According to the Riemann
vanishing theorem, it is represented in the form
\[
\mathfrak{A}_i=  (2\omega)^{-1}  \int_{(\infty,\infty)}^{(e_i,0)}
\boldsymbol{u}+\boldsymbol{K}_{P_0},
\]
where $\boldsymbol{K}_{P_0}$ is the vector of Riemann constants
with base at $P_0=(\infty,\infty)$.

The 10 even half-periods are represented in the form
\[
\mathfrak{A}_{i,j}=  (2\omega)^{-1}\left(
  \int_{(\infty,\infty)}^{(e_i,0)} \boldsymbol{u} +
  \int_{(\infty,\infty)}^{(e_j,0)} \boldsymbol{u}
\right)+\boldsymbol{K}_{P_0}, \quad 1\leq i < j \leq 5
\]
and one can denote the corresponding even characteristic as
$[\varepsilon_{i,j}]$.

There are formulae due to Bolza \cite{bol886}, see also \cite{ehkkls12},
which express the branch points $e_i$ in terms of derivative
$\theta$-constants and also which find the correspondence between
branch points and odd characteristics
\begin{equation}
  e_i\leftrightarrow [\delta_i] :\qquad
  e_i=- \frac{\partial_{\boldsymbol{U}} \theta[\delta_i] }
  {\partial_{\boldsymbol{V}} \theta[\delta_i]}, \quad i=1,\ldots,5.
\end{equation}

Because $e_6=\infty$, the characteristic $[\delta_6]$ is the
characteristic of the vector of Riemann constants,
$[\gamma]=[\delta_6]$. Therefore only 5 half-periods
$\boldsymbol{\mathfrak{A}}_i$ satisfy the condition (\ref{condition}).

We emphasise, that the procedure described here allows us to find the
vector of Riemann constants and the correspondence between half
periods and branch points, using Maple/algcurves software.  It is not
necessary to visualise the homology basis generated by the
Tretkoff-Tretkoff algorithm and enciphered in the Maple program.  Our
procedure allows us to represent the matrix $\varkappa= 2\eta
\omega^{-1}$ in the following form \cite{ehkkls12}:
\begin{prop}
  For any pair of integers $i,j$ $1\leq i < j \leq 5$ one can write 10
  relations for each of the 10 even characteristic $[\varepsilon_{i,j}]$
\begin{align}
\begin{split}
  \varkappa= &-\frac12 \left(\begin{array}  {cc}
      e_ie_j(e_k+e_m+e_n)+e_ke_me_n&-e_ie_j \\
      -e_ie_j& e_i+e_j\end{array} \right)\\&-\frac12
  {(2\omega)^{-1}}^T \frac{1}{\theta[\varepsilon_{i,j}]}\left(
    \begin{array}{cc} \theta_{1,1}[\varepsilon_{i,j}] &
      \theta_{1,2}[\varepsilon_{i,j}] \\\\
      \theta_{1,2}[\varepsilon_{i,j}]&\theta_{2,2}[\varepsilon_{i,j}]
   \end{array}
  \right)(2\omega)^{-1}
\end{split} \label{kappaformula1}
\end{align}
with $k\neq m\neq n \neq i \neq j \in \{1,\ldots, 5\}$ and
$\theta_{r,s}[\varepsilon] = \partial^2 \theta[\varepsilon] / \partial
z_{r} \partial z_s $.
\end{prop}
\begin{proof}
Let
\[
\wp_{mn}(\boldsymbol{z}) =-\frac{\partial^2}{\partial z_m\partial
  z_n} \, \mathrm{ln} \, \sigma(\boldsymbol{z}), \quad m,n=1,2 .
\]
Let $\boldsymbol{z}^{(i,j)}$ be the Abelian image of two branch
points, $e_i,e_j$
\[
\boldsymbol{z}^{(i,j)}=  \int_{P_0}^{(e_i,0)}
\boldsymbol{u}  + \int_{P_0}^{(e_j,0)} \boldsymbol{u}.
\]
Then Baker's solution of the Jacobi inversion problem (\ref{JIP})
leads to the equalities
\begin{align*}
  \wp_{22}(\boldsymbol{z}^{(i,j)})&=e_i+e_j,\\
  \wp_{12}(\boldsymbol{z}^{(i,j)})&=-e_ie_j,\\
  \wp_{11}(\boldsymbol{z}^{(i,j)})&=\frac{F(e_i,e_j)}{4(e_i-e_j)^2}=
  e_ie_j(e_k+e_m+e_n)+e_ke_me_n,
\end{align*}
where $i\neq j \neq k \neq m \neq n \in \{1,\ldots, 5\}$.  Taking into
account the definition of the $\sigma$-function (\ref{sigmafund}) in
terms of $\theta$-functions, as well the $\theta$-relation,
(\ref{thetarelation}), we find
\begin{align*}
  \wp_{11}(\boldsymbol{z}^{(i,j)})=
  -2\varkappa_{1,1}-\frac{\partial_{\boldsymbol{U}}^2\theta_{1,1} [
    \varepsilon_{i,j}]} { \theta[\varepsilon_{i,j}] }
\end{align*}
and similar expressions for $\wp_{12}(\boldsymbol{z}^{(i,j)})$,
$\wp_{22}(\boldsymbol{z}^{(i,j)})$.  Solving the above equations with
respect to $\varkappa_{i,j}$, we get (\ref{kappaformula1}).
\end{proof}

The formulae (\ref{kappaformula1}) represent the generalization of the
Weierstrass formulae
\begin{align}
2\eta\omega=-2e_1\omega^2-\frac12\frac{\vartheta_2''}{\vartheta_2}
=-2e_2\omega^2-\frac12\frac{\vartheta_3''}{\vartheta_3}=-2e_3\omega^2
-\frac12\frac{\vartheta_4''}{\vartheta_4}
\end{align}
to the genus two hyperelliptic curve.  Recall that in the Weierstrass theory,
$e_1+e_2+e_3=0$, so adding these three formulae gives the first of
(\ref{Weierstrass1}).

For typographical convenience, we will use a shorter notations for
directional derivatives
\begin{align*} &\partial_{\boldsymbol{U}}
  \theta[\varepsilon]=\Theta_1[\varepsilon],\quad \partial_{\boldsymbol{V}}
  \theta[\varepsilon]=\Theta_2[\varepsilon],\quad \partial_{\boldsymbol{U}^2}
  \theta[\varepsilon]=\Theta_{1,1}[\varepsilon],\\
  &\partial_{\boldsymbol{U}\boldsymbol{V}}
  \theta[\varepsilon]=\Theta_{1,2}[\varepsilon],\quad \partial_{\boldsymbol{V}^2}
  \theta[\varepsilon]=\Theta_{2,2}[\varepsilon],\quad
  \text{etc.}
\end{align*}
also, $\theta[\varepsilon]=\Theta[\varepsilon]$ at even
$[\varepsilon]$.  Summing (\ref{kappaformula1}) over 10 even
characteristics, and using the above notation for directional
derivatives, we get
\begin{equation}
  \varkappa =\frac{1}{80} \left(\begin{array}{cc}
      4\lambda_2 & \lambda_3 \\
      \lambda_3 & 4\lambda_4
    \end{array}  \right) -\frac{1}{20}
  \sum_{ 10 \;\; \rm{even}\;\; [\varepsilon]} \frac{1}{\Theta[\varepsilon]}
  \left(\begin{array}{cc}
      \Theta_{1,1}[\varepsilon]
      &  \Theta_{1,2}[\varepsilon]\\
\\  \Theta_{1,2}[\varepsilon]&
   \Theta_{2,2}[\varepsilon]
\end{array}\right).\label{kappaformula2}
\end{equation}

The representation (\ref{kappaformula2}) can be compared to the
formula of Korotkin-Schramchenko \cite{ksh12}, derived for a general
algebraic curve.  We emphasise that (\ref{kappaformula2}) is written
in the Baker basis (\ref{bakerbasis}), which enables us to write, in
simpler form, the differential equations for the $\wp$-symbols, the
multi-dimensional generalisations of the Weierstrass $\wp$-function.

Now we are in the position to present generalisations of the
Weierstrass formula (\ref{Weierstrass1}) to genus two curves. All the
following relations are obtained by the expansion procedure described
above, exemplified in the case of the elliptic curve.
\begin{prop}
Denote according to the Bolza formula
\[
\frac{  \Theta_1[\delta_i]}
{  \Theta_2[\delta_i]}=-e_i
,\]
where $e_1,\ldots,e_5$ are branch points of the curve $\mathcal{C}$.
Then the entries to the matrix $\varkappa$ have the form
\begin{align}
  \varkappa_{2,2} &=\frac{1}{24}\lambda_4-\frac16 e_i
  -\frac1{6}\frac{\Theta_{2,2,2}[\delta_i]}
   {\Theta_2[\delta_i]},\label{kappa22}\\
  \varkappa_{1,2} &=-\frac{1}{24} \lambda_4 e_i-\frac13
  e_i^2-\frac{1}{12}e_i
  \frac{ \Theta_{2,2,2}[\delta_i]}
  {\Theta_2[\delta_i]}
  -\frac1{4}\frac{
    \Theta_{1,2,2}[\delta_i]} { \Theta_2[\delta_i]},
\label{kappa12}\\
\begin{split}
  \varkappa_{1,1} &= -\frac{\lambda_3}8e_i-\frac{5}{24}\lambda_4
  e_i^2-\frac76e_i^3 -\frac12
  \frac{\Theta_{1,1,2}[\delta_i]}
  {\Theta_2[\delta_i]}-\frac12e_i
  \frac{\Theta_{1,2,2}[\delta_i]}
  { \Theta_2[\delta_i]}\\&-\frac{1}{6}e_i^2
  \frac{ \Theta_{2,2,2}[\delta_i]}
  {\Theta_2[\delta_i]}.
       \end{split}
  \label{kappa11}
\end{align}
\end{prop}
\begin{proof}
  The proof is based on the expansion procedure of (\ref{master}) with
  subsequent elimination of matrix elements $\varkappa_{i,j}$ from the
  relations obtained. We omit here rather cumbersome computer algebra
  details.
\end{proof}

Summing up over all 5 representations of $\varkappa_{2,2}$, we get
\begin{equation}
  \varkappa_{2,2}=\frac{1}{20}\lambda_4 - \frac{1}{30}
  \sum_{ 5 \;\; \rm{odd} \;\; [\delta] }
  \frac{
    \Theta_{2,2,2}[\delta]}{\Theta_2[\delta]},
\label{kappaformula3}
\end{equation}
where we exclude from the summation the characteristic $[\delta_6]$ of
the vector of Riemann constants for which $
\Theta_2[\gamma]=0$.  Analogously for $\varkappa_{1,2}$ and
$\varkappa_{1,1}$ we get
\begin{align} \begin{split} \varkappa_{1,2}& =\frac{1}{40}\lambda_3
    -\frac{1}{800}\lambda_4^2 -\frac{1}{20}\sum_{ 5 \;\; \rm{odd} \;\;
      [\delta] }
      \frac{\Theta_{1,2,2}[\delta]}
    { \Theta_2[\delta]}+\frac{1}{1200}\lambda_4\sum_{ 5 \;\; \rm{odd} \;\; [\delta] }
    \frac{
      \Theta_{2,2,2}[\delta]}{\Theta_2[\delta]}
\end{split} \label{kappa12sum}
 \end{align}
and
\begin{align} \begin{split} \varkappa_{1,1}&=\frac{3}{40}\lambda_2
    -\frac{1}{400}\lambda_4\lambda_3 +\frac{1}{8000}\lambda_4^3
    -\frac{1}{10} \sum_{ 5 \;\; \rm{odd} \;\; [\delta] }
    \frac{
      \Theta_{1,1,2}[\delta]}{
      \Theta_2[\delta]}\\
    & +\frac{1}{200}\lambda_4\sum_{ 5 \;\; \rm{odd} \;\; [\delta] }
    \frac{ \Theta_{1,2,2}[\delta]}
    { \Theta_2[\delta]}
    -\frac{1}{12000}\lambda_4^2\sum_{ 5 \;\; \rm{odd} \;\; [\delta] }
    \frac{
      \Theta_{2,2,2}[\delta]}{ \Theta_2[\delta]}.
\end{split} \label{kappa11sum}
 \end{align}

When $\lambda_4=0$, the $\varkappa$-matrix takes the simpler form
\begin{align}
\begin{split}
  \varkappa&=\frac{1}{40}\left(  \begin{array}{cc} 3\lambda_2&\lambda_3\\
      \lambda_3&0 \end{array}\right) \\ & \qquad -\frac{1}{20}\sum_{ 5
    \;\; \rm{odd} \;\; [\delta] } \frac{1}{
    \Theta_2[\delta]} \left( \begin{array}{cc}
      2 \Theta_{1,1,2}[\delta]&
       \Theta_{1,2,2}[\delta]\\\\
      \Theta_{1,2,2}[\delta]&
      \frac23 \Theta_{2,2,2}[\delta]
 \end{array} \right).
\end{split}
\end{align}

\subsection{Certain $\theta$-constant relations }
Comparing the formulae (\ref{kappaformula2}) derived in
\cite{ehkkls12}, and the above formula, we conclude
\begin{prop} Let $\mathcal{C}$ be the genus two hyperelliptic curve
  with branch point at infinity and realised in the form
\[
y^2=4x^5+\lambda_4x^4+\lambda_3x^3+\lambda_2x^2+\lambda_1x+\lambda_0,
\quad \lambda_i\in \mathbb{C}.
\]
Let $2\omega$ be the matrix of $\mathfrak{a}$-periods of holomorphic
differentials.  Then the
following relation holds
\begin{equation}
  \sum_{ 5 \;\; \rm{odd} \;\; [\delta] }
  \frac{
\Theta_{2,2,2}[\delta]}{\Theta_2[\delta]}=
\frac32\sum_{ 10 \;\; \rm{even}\;\; [\varepsilon]}
\frac{
\Theta_{2,2}[\varepsilon]}{ \Theta[\varepsilon]}.
\label{newformula}
\end{equation}
There exists necessarily one odd
characteristic $[\delta]$ for which
$\Theta_2[\delta]= 0$ and the summation on the left
hand side over the odd $[\delta]$ runs over the remaining 5.
\end{prop}
This is a generalization of the Weierstrass formula,
\[
\frac{\vartheta_1'''(0)}{\vartheta_1'(0)} =
\frac{\vartheta_2''(0)}{\vartheta_2(0)}
+\frac{\vartheta_3''(0)}{\vartheta_3(0)}  +
\frac{\vartheta_4''(0)}{\vartheta_4(0)}   .
\]

Comparing in the same way the expressions (\ref{kappaformula2}) and
(\ref{kappa12sum}), (\ref{kappa11sum}), we find that when $\lambda_4=0$:
\begin{equation}
  4\sum_{ 5 \;\; \rm{odd} \;\; [\delta] }
  \frac{
    \Theta_{1,2,2}[\delta]}{ \Theta_2[\delta]}
  -4\sum_{ 10 \;\; \rm{even}\;\; [\varepsilon]}
  \frac{
    \Theta_{1,2}[\varepsilon]}{ \Theta[\varepsilon]} =
  \lambda_3\label{newformula2}
\end{equation}
and
\begin{equation}
  4\sum_{ 5 \;\; \rm{odd} \;\; [\delta] }
  \frac{
    \Theta_{1,1,2}[\delta]}{ \Theta_2[\delta]}
  -2\sum_{ 10 \;\; \rm{even}\;\; [\varepsilon]}
  \frac{
    \Theta_{1,1}[\varepsilon]}{ \Theta[\varepsilon]}=\lambda_2.
\label{newformula3}
\end{equation}

The formulae (\ref{newformula2}) and (\ref{newformula3}) express the
parameters of the curve, in this case - the symmetric combination of
branch points - in terms of sums of theta constants which are
symmetric with respect to characteristics.  These formulae can be
interpreted as a new kind of Thomae-type formulae.  The derivation of
such classes of relations in the case of non-hyperelliptic curves
would be of interest.

The above formulae and (\ref{kappa22}) - (\ref{kappa11}) lead to
various generalisations of the Jacobi derivative formula.
E.g. subtracting the two expressions (\ref{kappa22}) written for
different indices $i$ and $j$, we get (using the classical {\it
  Rosenhain derivative formula} \cite{ros851} for simplifications)
\begin{align}
\begin{split}
 & \pm \pi^2 \mathrm{det} (2\omega)^{-1} \Theta[\varepsilon_p]
  \Theta[\varepsilon_q] \Theta[\varepsilon_r] \Theta[\varepsilon_s]\\
  &\qquad \qquad=
  \Theta_{2,2,2}[\delta_i]
  \Theta_2[\delta_j]- \Theta_{2,2,2}[\delta_j]
 \Theta_2[\delta_i],
\end{split} \label{genjacder}
\end{align}
where $[\delta_i], [\delta_j]$ are two arbitrary odd characteristics
from the set of $[\delta_1], \ldots, [\delta_6]$, and 4 even
characteristics, $[\varepsilon_p], [\varepsilon_q], [\varepsilon_r],
[\varepsilon_s]$ are of the form $[\delta_i]+[\delta_j]+[\delta_k] \;
\mathrm{mod}\;2$ where $k\in \{1,2,3,4,5,6\} / \{i,j\}$.  This formula
can be interpreted as a {\it Higher Rosenhain derivative formula}.

New interesting generalisations of the Jacobi derivative formula were
recently found by Grushevsky and Salvati Manni \cite{gm05}, who also
presented a detailed list of references to the other known
generalizations in their paper.  We do not discuss here the relevance
of the formulae obtained here to the results \cite{gm05} but plan to
consider this question in a separate publication.

Concluding, we note that the procedures described here, of the
derivation of formulae of the form (\ref{kappa22}) - (\ref{kappa11}),
works in all cases when the Klein-Weierstrass algebraic representation
of the bi-differential $\Omega(Q,R)$ is known.  Therefore the next cases
that could be analyzed are the cases $(3,s)$ - trigonal curves,
\begin{equation}
y^3-a_2(x)y^2-a_1(x)y-a_0(x)=0
\end{equation}
with appropriate polynomials $a_1(x)$ and $a_0(x)$. Analytic expressions
for the basic meromorphic differentials can be found in \cite{bel12},
\cite{eemop08}, whilst expressions for the projective connection $S_{KW}(P)$
are given in the course of the proof of Prop. 2.1.

{\bf Acknowledgements} Two of the authors, JCE and VZE, are grateful
to S.\ Grushevky for a stimulating discussion on the possibilities of
deriving Thomae-type formulae that express parameters of the curves,
such as coefficients of defining polynomial $\lambda$, in terms of
$\theta$-constants during a meeting in ICMS, Edinburgh, in October
2012. The authors are grateful to V.\ Buchstaber, S.\ Grushevsky, D.\
Korotkin, and A. \ Nakayashiki who agree to read a draft of this paper
before its publication and made remarks that we took into account in
the final version.  KE and VZE gratefully acknowledges the Deutsche
Forschungsgemeinschaft (DFG) for financial support within the
framework of the DFG Research Training group 1620 Models of gravity.
The work of VZE was supported by the School of Mathematics, University
of Edinburgh, under the certificate of sponsorship C5E7V94128U.

\providecommand{\bysame}{\leavevmode\hbox to3em{\hrulefill}\thinspace}
\providecommand{\MR}{\relax\ifhmode\unskip\space\fi MR }
\providecommand{\MRhref}[2]{%
  \href{http://www.ams.org/mathscinet-getitem?mr=#1}{#2}
}
\providecommand{\href}[2]{#2}

\end{document}